\documentclass[preprint,12pt,3p,times]{elsarticle}
\usepackage{amsfonts,amsmath,amsthm}
\usepackage{amssymb}

\usepackage{color}

\usepackage{amsmath, subfigure}

\newcommand{\be}{\begin{eqnarray*}}                           
\newcommand{\en}{\end{eqnarray*}}
\newcommand{\bes}{\begin{eqnarray}}                           
\newcommand{\ens}{\end{eqnarray}}

\newtheorem{theorem}{Theorem}[section]

\newtheorem{lemma}{Lemma}[section]

\def\bq{\begin{equation}}
\def\eq{\end{equation}}
\def\bqq{\begin{eqnarray*}}
\def\eqq{\end{eqnarray*}}

\begin{document}

\begin{frontmatter}

\title{{ On an inverse problem in
the parabolic equation arising from groundwater pollution problem }}

\author[Affil1]{Thanh Binh Tran }
\author[Affil2]{Huy Tuan Nguyen\corref{cor1}}
\author[Affil3]{Van Thinh Nguyen}
\author[Affil4]{Vo Anh Khoa}

\address[Affil1]{Department of Mathematics and Applications, Sai Gon University, 
Ho Chi Minh, Vietnam}
\address[Affil2]{Faculty of Mathematics and
Statistics, Ton Duc Thang University,
 Ho Chi Minh city,
Vietnam.}
\address[Affil4]{Mathematics and Computer Science Division, Gran Sasso Science Institute, Viale Francesco Crispi 7, 67100 L'Aquila, Italy.}
\address[Affil3]{Department of Civil and Environmental Engineering, Seoul National University, Republic of Korea}

\cortext[cor1]{Corresponding Author: Nguyen Huy Tuan; Email: thnguyen2683@gmail.com;}

\begin{abstract}
{In this paper, we consider an inverse problem to determine a
\textcolor{red}{source term} in a parabolic equation, where the data are obtained at a certain time.
In general, this problem is ill-posed, therefore the Tikhonov regularization
method is proposed to solve the problem. In the theoretical results, a priori error estimate between the exact solution and its regularized solution is obtained. We also propose both methods, a priori and a posteriori parameter choice rules. In addition, the proposed methods have been verified by numerical experiments to estimate the errors between the regularized solutions and exact solutions. Eventually, from the numerical results it shows that the a posteriori parameter choice rule method gives a better the convergence speed in comparison with the a priori parameter choice rule method in some specific applications. }
\end{abstract}
\end{frontmatter}

\section{Introduction}
Groundwater is crucial to human being, environment and economy, because a large portion of drinking water comes from groundwater, and it is extracted for commercial, industrial and irrigation uses. Groundwater also sustains stream flow during dry periods, and is critical to the function of streams, wetlands and other aquatic environments. Therefore, protecting the safety and security of groundwater is essential for communities, and for the environment.
In recent years, mathematical models have \textcolor{red}{been} used to analyze ground water system. There are two notable approaches in dealing with groundwater modeling, the forward and backward approaches. The former is going to predict unknown variables by solving appropriate governing equations, while the latter is going to determine unknown physical parameters. 
Most of groundwater models are distributed parameter models, where the parameters used in the modeling equations are not directly obtained from physical observations, but from trial-and-error and graphical fitting techniques. If large errors are included in mathematical model structure, model parameters, sink/source terms and boundary conditions, the model cannot produce accurate results. To deal with this issue, the inverse problem of parameter identification has been applied. In groundwater applications such as in finding a previous pollution source intensity from observation data of the pollutant concentrations at a later time, or in designing the final state of melting and freezing processes, it is necessary to construct a heat source at any given time from the final outcome state data. The groundwater inverse problem has been studied since the middle of 1970s by McLaughin (1975), Yeh (1986), Kuiper (1986), Carrera (1987), Ginn and Cushman (1990) and Sun (1994), etc (see in \cite{Mc,Ye,Ca,GC,Ku,Su}). Some remarkable results on this research area should be mentioned by McLaughlin and Townley (1996) \cite{MT} and Poeter and Hill (1997) \cite{PH}. Under consideration of a solute diffusion, the flow and self-purifying function of watershed system, the concentration of pollution
$u(x,t)$  at any time in a watershed is described by \textcolor{red}{the} following one-dimensional linear parabolic equation:
\begin{equation}
\frac{\partial u}{\partial t}-\eta\frac{\partial^{2}u}{\partial x^{2}}+\nu\frac{\partial u}{\partial x}+\gamma u=P\left(x,t\right),\quad x\in\Omega,t>0,\label{eq:1}
\end{equation}
where $\Omega\in\mathbb{R}$ is the spatial studied domain, 
$\eta$ is the diffusion coefficient, $\nu$ is mean velocity of water in the 
watershed, and $\gamma$ is the self-purifying
function of the watershed, $P\left(x,t\right)$ is the source
term causing the pollution function $u\left(x,t\right)$.
By setting
\[
w\left(x,t\right)=u\left(x,t\right)\mbox{exp}\left(\dfrac{\nu}{2\eta}x-\left(\dfrac{\nu^{2}}{4\eta}+\gamma\right)t\right)
\]
and 
\[F\left(x,t\right)=P\left(x,t\right)\mbox{exp}\left(\left(\dfrac{\nu^{2}}{4\eta}+\gamma\right)t-\dfrac{\nu}{2\eta}x\right),
\]
Then Eq. (\ref{eq:1}) becomes
\[
\frac{\partial w}{\partial t}-\eta\frac{\partial^{2}w}{\partial x^{2}}=F\left(x,t\right).
\]
This equation is well-known \textcolor{red}{to be the} parabolic heat equation with
time-dependent coefficients.  
This equation has been investigated for the heat source with either temporal
\cite{FL,JL,YF}, or \textcolor{red}{spatial-dependent} \cite{AB,CF,Sa,YF1,YF2} only.
There are few studies on identification of the source term depending on both time and space in term of a separable form of $F(x,t)$, as  $F(x,t)=\varphi(t)f(x)$; where
$\varphi\left(t\right)$ is a given function; for instance Hassanov \cite{Ha}  identified the heat source in the form
of $F(x,t)=F(x)H(t)$ for the variable coefficient heat conduction
equation; $u_{t}=(k(x)u_{x})_{x}+F(x)H(t)$ under the variational method. However, in the case with the time-dependent coefficient of ${\displaystyle \frac{\partial^{2}w}{\partial x^{2}}}$, there are still limited results.\\
In this study, we consider the equation for groundwater pollution as follows:
\begin{equation}
\frac{\partial u}{\partial t}-\frac{\partial}{\partial x}\left(a\left(t\right)\frac{\partial u}{\partial x}\right)=\varphi\left(t\right)f\left(x\right),\quad\left(x,t\right)\in\left(0,\pi\right)\times\left(0,T\right),\label{eq:2}
\end{equation}
with initial and final conditions
\begin{equation}
u\left(x,0\right)=0,\quad u\left(x,T\right)=g\left(x\right),\quad x\in\left(0,\pi\right),\label{eq:3}
\end{equation}
and boundary condition
\begin{equation}
u\left(0,t\right)=u\left(\pi,t\right)=0.\label{eq:4}
\end{equation}
Here, $a\left(t\right)>0$, $g\left(x\right)$ and $\varphi\left(t\right)$
are given functions. In this paper, we will determine the source
term $f\left(x\right)$ from the inexact observed data of $\varphi\left(t\right)$
and $g\left(x\right)$.\\
Let $\left\Vert .\right\Vert $ and $\left\langle .,.\right\rangle $
be the norm and the inner product in $L^{2}\left(0,\pi\right)$, respectively.
Now, we take an orthonormal basis in $L^{2}\left(0,\pi\right)$ satisfying
the boundary condition (\ref{eq:4}), particularly the basic function $\sqrt{\dfrac{2}{\pi}}\sin\left(nx\right)$ for $n\in\mathbb{N}$ \textcolor{red}{satisfies that} condition. Then, by an elementary calculation the problem
(\ref{eq:2}) under the conditions (\ref{eq:3}) and (\ref{eq:4}) can be transformed
into \textcolor{red}{the} following corresponding problem 
\begin{equation}
\frac{d}{dt}\left\langle u\left(x,t\right),\sin\left(nx\right)\right\rangle +n^{2}a\left(t\right)\left\langle u\left(x,t\right),\sin\left(nx\right)\right\rangle =\varphi\left(t\right)\left\langle f\left(x\right),\sin\left(nx\right)\right\rangle ,\quad t\in\left(0,T\right),\label{eq:5}
\end{equation}
\begin{equation}
\left\langle u\left(x,0\right),\sin\left(nx\right)\right\rangle =0,\quad\left\langle u\left(x,T\right),\sin\left(nx\right)\right\rangle =\left\langle g\left(x\right),\sin\left(nx\right)\right\rangle .\label{eq:6}
\end{equation}
By setting ${\displaystyle A\left(t\right)=\int_{0}^{t}a\left(s\right)ds}$,
we can solve the ordinary differential equation (\ref{eq:5}) with
the contitions (\ref{eq:6}). We thus obtain
\begin{equation}
\left\langle f\left(x\right),\sin\left(nx\right)\right\rangle =e^{n^{2}A\left(T\right)}\left(\int_{0}^{T}e^{n^{2}A\left(t\right)}\varphi\left(t\right)dt\right)^{-1}\left\langle g\left(x\right),\sin\left(nx\right)\right\rangle ,\label{eq:7}
\end{equation}
which leads to
\begin{equation}
f\left(x\right)=\sum_{n=1}^{\infty}e^{n^{2}A\left(T\right)}\left(\int_{0}^{T}e^{n^{2}A\left(t\right)}\varphi\left(t\right)dt\right)^{-1}g_{n}\sin\left(nx\right),\label{eq:8}
\end{equation}
where $g_{n}=\dfrac{2}{\pi}\left\langle g\left(x\right),\sin\left(nx\right)\right\rangle $.\\
Note that $e^{n^{2}A\left(T\right)}$ increases rather quickly once
$n$ becomes large. Thus, the exact data function $g\left(x\right)$
must satisfy that $g_{n}$ decays at least as the same speed of $e^{n^{2}A\left(t\right)}$. However, in application the
input data $g\left(x\right)$ from observations will never be exact due to the measurements. We assume the data functions $g_{\epsilon}\left(x\right)\in L^{2}\left(0,\pi\right)$,
and $\varphi\left(t\right),\varphi_{\epsilon}\left(t\right)\in L^{2}\left(0,T\right)$
satisfy
\begin{equation}
\left\Vert g_{\epsilon}-g\right\Vert \le\epsilon,\quad\left\Vert \varphi_{\epsilon}-\varphi\right\Vert\le\epsilon,
\end{equation}
where $\epsilon>0$ represents a noise from observations.

The main objective of this paper is to determine a conditional stability,
and provide the revised generalized Tikhonov regularization
method. In addition, the stability estimate between the regularization solution
and the exact solution is obtained. For explanation of this method,
we impose an a priori bound on the data
\begin{equation}
\left\Vert f\right\Vert _{H^{k}\left(0,\pi\right)}\le M,\quad k\ge0,
\end{equation}
where $M\ge0$ is a constant, and $\left\Vert .\right\Vert _{H^{k}\left(0,\pi\right)}$
denotes the norm in the Sobolev space $H^{k}\left(0,\pi\right)$ of
order $k$ can be naturally defined in terms of Fourier series whose
coefficients decay rapidly; namely,
\begin{equation}
H^{k}\left(0,\pi\right):=\left\{ f\in L^{2}\left(0,\pi\right):\left\Vert f\right\Vert _{H^{k}\left(0,\pi\right)}<\infty\right\} ,
\end{equation}
equipped with the norm
\[
\left\Vert f\right\Vert _{H^{k}\left(0,\pi\right)}=\sqrt{\sum_{n=1}^{\infty}\left(1+n^{2}\right)^{k}f_{n}^{2}},
\]
where $f_{n}$ defined by $f_{n}=\left\langle f,X_{n}\right\rangle ,X_{n}=\sqrt{\dfrac{2}{\pi}}\sin\left(nx\right)$
is the Fourier coefficient of $f$.\\
As a regularization method, the Tikhonov method has been used to solve
ill-posed problems in a number of publications. However, most of previous  works focus on an a priori choice of the regularization parameter. There is usually a defect in any a priori method; i.e. the a priori choice of the regularization parameter depends obviously
on the a priori bound $M$ of the unknown solution. In fact,
the a priori bound $M$ cannot be known exactly in practice, and working
with a wrong constant $M$ may lead to a bad regularization solution.
In this paper, we mainly consider the a posteriori choice of
a regularization parameter for the mollification method. Using the
discrepancy principle we provide a new posteriori parameter choice rule. \\

The outline of this paper is as follows. In Section 2, a conditional stability is introduced. A Tikhonov regularization and its convergence under an a priori parameter choice rule is presented in Section 3. Similarly to Section 3, another Tikhonov regularization and its convergence under a posteriori parameter choice rule is shown in Section 4. In Section 5, we introduce two numerical examples, which are implemented from proposal regularization methods, the numerical results are compared with exact solutions.  
\section{A conditional stability}
Let $a,\varphi,\varphi_{\epsilon}:\left[0,T\right]\to\mathbb{R}$
be continuous functions. We suppose that there exist constants $B_{1},B_{2},C_{1},C_{2},D_{1},D_{2}>0$
such that
\begin{equation}
B_{1}\le\varphi\left(t\right)\le B_{2},\quad C_{1}\le\varphi_{\epsilon}\left(t\right)\le C_{2},\quad D_{1}\le a\left(t\right)\le D_{2}.\label{eq:12}
\end{equation}
Hereafter, let us set
\begin{equation}
A\left(t\right)-A\left(T\right)=B\left(t\right),\quad\Phi\left(n,h\right)=\int_{0}^{T}e^{n^{2}B\left(t\right)}h\left(t\right)dt,\label{eq:13}
\end{equation}
where $h$ \textcolor{red}{plays a role as} $\varphi$, and $\varphi_{\epsilon}$ by implication.
Then, we can obtain the following conditional stability.
\begin{lemma}
For all continuous functions $h\in\left[E_{1},E_{2}\right]$, then 
\begin{equation}
\left(\Phi\left(n,h\right)\right)^{k}\le\begin{cases}
E_{2}^{k}n^{-2k}D_{1}^{-k}\left(1-e^{-n^{2}D_{2}T}\right)^{k} & ,k\ge0,\\
E_{1}^{k}n^{-2k}D_{2}^{-k}\left(1-e^{-D_{1}T}\right)^{k} & ,k<0,
\end{cases}
\end{equation}
for $n\in\mathbb{N}$.\end{lemma}
\begin{proof}
The proof is simple by elementary calculation.\end{proof}
\begin{theorem}
If there exists $M\ge0$ such that $\left\Vert f\right\Vert _{H^{k}\left(0,\pi\right)}\le M$,
then
\begin{equation}
\left\Vert f\right\Vert \le\left(\frac{D_{2}}{B_{1}\left(1-e^{-D_{1}T}\right)}\right)^{\frac{k}{k+2}}M^{\frac{2}{k+2}}\left\Vert g\right\Vert ^{\frac{k}{k+2}}.
\end{equation}
\end{theorem}
\begin{proof}
Using Holder's inequality, we first have
\begin{eqnarray}
\left\Vert f\right\Vert ^{2} & = & \sum_{n=1}^{\infty}\left(\int_{0}^{T}e^{n^{2}B\left(t\right)}\varphi\left(t\right)dt\right)^{-2}g_{n}^{\frac{4}{k+2}}g_{n}^{\frac{2k}{k+2}}\nonumber \\
 & \le & \left[\sum_{n=1}^{\infty}\left(\int_{0}^{T}e^{n^{2}B\left(t\right)}\varphi\left(t\right)dt\right)^{-\left(k+2\right)}g_{n}^{2}\right]^{\frac{2}{k+2}}\left(\sum_{n=1}^{\infty}g_{n}^{2}\right)^{\frac{k}{k+2}}.
\end{eqnarray}
\textcolor{red}{Then} from (\ref{eq:7}), the inequality \textcolor{red}{becomes}
\begin{equation}
\left\Vert f\right\Vert ^{2}\le\left[\sum_{n=1}^{\infty}\left(\int_{0}^{T}e^{n^{2}B\left(t\right)}\varphi\left(t\right)dt\right)^{-k}f_{n}^{2}\right]^{\frac{2}{k+2}}\left\Vert g\right\Vert ^{\frac{2k}{k+2}}.
\end{equation}
We pay attention to the integral on the right-hand side by direct
estimate and computation. From (\ref{eq:12}), we thus get
\begin{eqnarray}
\left\Vert f\right\Vert ^{2} & \le & \left(\frac{D_{2}^{k}}{\left[B_{1}\left(1-e^{-D_{1}T}\right)\right]^{k}}\right)^{\frac{2}{k+2}}\left(\sum_{n=1}^{\infty}n^{2k}f_{n}^{2}\right)^{\frac{2}{k+2}}\left\Vert g\right\Vert ^{\frac{2k}{k+2}}\nonumber \\
 & \le & \left[\frac{D_{2}}{B_{1}\left(1-e^{-D_{1}T}\right)}\right]^{\frac{2k}{k+2}}\left\Vert f\right\Vert _{H^{k}\left(0,\pi\right)}^{\frac{4}{k+2}}\left\Vert g\right\Vert ^{\frac{2k}{k+2}},
\end{eqnarray}
\textcolor{red}{Hence, the theorem} has been proved.
\end{proof}
\section{Tikhonov regularization under an a priori parameter choice rule}
Define a linear operator $K:L^{2}\left(0,\pi\right)\to L^{2}\left(0,\pi\right)$
as follows.
\begin{equation}
Kf\left(x\right)=\sum_{n=1}^{\infty}\left\langle f,X_{n}\right\rangle \int_{0}^{T}e^{n^{2}B\left(t\right)}dtX_{n}\left(x\right)=\int_{0}^{\pi}k\left(x,\xi\right)f\left(\xi\right)d\xi,\label{eq:19}
\end{equation}
where ${\displaystyle k\left(x,\xi\right)=\sum_{n=1}^{\infty}\int_{0}^{T}e^{n^{2}B\left(t\right)}dtX_{n}\left(x\right)X_{n}\left(\xi\right)}$.
Due to $k\left(x,\xi\right)=k\left(\xi,x\right)$, $K$ is self-adjoint.
Next, we prove its compactness.
\textcolor{red}{Let us consider} finite rank operators $K_{m}$ by
\begin{equation}
K_{m}f\left(x\right)=\sum_{n=1}^{m}\left\langle f,X_{n}\right\rangle \int_{0}^{T}e^{n^{2}B\left(t\right)}dtX_{n}\left(x\right).\label{eq:20}
\end{equation}
Then, from (\ref{eq:19}) and (\ref{eq:20}), we have
\begin{equation}
\left\Vert K_{m}f-Kf\right\Vert ^{2}=\sum_{n=m+1}^{\infty}\left(\int_{0}^{T}e^{n^{2}B\left(t\right)}dt\right)^{2}f_{n}^{2}\le\frac{1}{m^{4}D_{1}^{2}}\sum_{n=m+1}^{\infty}f_{n}^{2}\le\frac{1}{m^{4}D_{1}^{2}}\left\Vert f\right\Vert ^{2}.
\end{equation}
Therefore, $\left\Vert K_{m}-K\right\Vert \to0$ in the sense of operator
norm in $\mathcal{L}\left(L^{2}\left(0,\pi\right);L^{2}\left(0,\pi\right)\right)$, as $m\to\infty$. $K$ is also a compact operator.
Next, the singular values for the linear self-adjoint compact operator
are
\begin{equation}
\sigma_{n}=\int_{0}^{T}e^{n^{2}B\left(t\right)}dt,
\end{equation}
and corresponding eigenvectors \textcolor{red}{are} $X_{n}$ known as an orthonormal
basis in $L^{2}\left(0,\pi\right)$.
From (\ref{eq:19}), the inverse source problem introduced above
can be formulated as an operator equation.
\begin{equation}
\left(Kf\right)\left(x\right)=g\left(x\right).
\end{equation}
In general, \textcolor{red}{such a problem is ill-posed}. From the point of view, we
\textcolor{red}{aim at solving} it by using Tikhonov regularization method, \textcolor{red}{i.e. the study of minimizing}
the \textcolor{red}{following} quantity in $L^{2}\left(0,\pi\right)$
\begin{equation}
\left\Vert Kf-g\right\Vert ^{2}+\mu^{2}\left\Vert f\right\Vert ^{2}.
\end{equation}
As shown in \cite{Ki} by Theorem 2.12, its minimizer $f_{\mu}$ satisfies
\begin{equation}
K^{*}Kf_{\mu}\left(x\right)+\mu^{2}f_{\mu}\left(x\right)=K^{*}g\left(x\right).
\end{equation}
Due to singular value decomposition for compact self-adjoint operator,
we have
\begin{equation}
f_{\mu}\left(x\right)=\sum_{n=1}^{\infty}\left(\mu^{2}+\left(\int_{0}^{T}e^{n^{2}B\left(t\right)}\varphi\left(t\right)dt\right)^{2}\right)^{-1}\int_{0}^{T}e^{n^{2}B\left(t\right)}\varphi\left(t\right)dt\left\langle g,X_{n}\right\rangle X_{n}\left(x\right).\label{eq:26}
\end{equation}
If the given data is noised, we can establish 
\begin{equation}
f_{\mu}^{\epsilon}\left(x\right)=\sum_{n=1}^{\infty}\left(\mu^{2}+\left(\int_{0}^{T}e^{n^{2}B\left(t\right)}\varphi_{\epsilon}\left(t\right)dt\right)^{2}\right)^{-1}\int_{0}^{T}e^{n^{2}B\left(t\right)}\varphi_{\epsilon}\left(t\right)dt\left\langle g_{\epsilon},X_{n}\right\rangle X_{n}\left(x\right).\label{eq:27}
\end{equation}
From (\ref{eq:13}), (\ref{eq:26}) and (\ref{eq:27}), we get
\begin{eqnarray}
f_{\mu}\left(x\right) & = & \sum_{n=1}^{\infty}\frac{\Phi\left(n,\varphi\right)}{\mu^{2}+\left(\Phi\left(n,\varphi\right)\right)^{2}}\left\langle g,X_{n}\right\rangle X_{n}\left(x\right),\\
f_{\mu}^{\epsilon}\left(x\right) & = & \sum_{n=1}^{\infty}\frac{\Phi\left(n,\varphi_{\epsilon}\right)}{\mu^{2}+\left(\Phi\left(n,\varphi_{\epsilon}\right)\right)^{2}}\left\langle g_{\epsilon},X_{n}\right\rangle X_{n}\left(x\right).
\end{eqnarray}
\textcolor{red}{In this work,} we will deduce an error estimate for $\left\Vert f-f_{\mu}^{\epsilon}\right\Vert $
and show convergence rate under a suitable choice of regularization parameters.
It is clear that the entire error can be decomposed into the bias
and noise propagation as follows:
\begin{equation}
\left\Vert f-f_{\mu}^{\epsilon}\right\Vert \le\left\Vert f-f_{\mu}\right\Vert +\left\Vert f_{\mu}-f_{\mu}^{\epsilon}\right\Vert .\label{eq:30}
\end{equation}
We first give the error bound for the noise term.
\begin{lemma}
\label{lem:3}If the noise assumption holds and assume that $\left\Vert g-g_{\epsilon}\right\Vert \le\epsilon$
and $\left\Vert \varphi-\varphi_{\epsilon}\right\Vert _{L^{2}\left[0,T\right]}\le\epsilon$,
then the solution depends continuously on the given data. Moreover, we
have the following estimate.
\begin{equation}
\left\Vert f_{\mu}-f_{\mu}^{\epsilon}\right\Vert \le\frac{\left\Vert f\right\Vert \left(\mu^{2}+D_{1}^{-2}B_{2}C_{2}\right)}{4\mu^{2}T^{2}B_{1}C_{1}}\left\Vert \varphi-\varphi_{\epsilon}\right\Vert _{L^{2}\left[0,T\right]}+\frac{1}{2\mu}\left\Vert g-g_{\epsilon}\right\Vert .
\end{equation}
\end{lemma}
\begin{proof}
We notice that
\begin{eqnarray}
f_{\mu}-f_{\mu}^{\epsilon} & = & \sum_{n=1}^{\infty}\left(\frac{\Phi\left(n,\varphi\right)}{\mu^{2}+\left(\Phi\left(n,\varphi\right)\right)^{2}}g_{n}X_{n}-\frac{\Phi\left(n,\varphi_{\epsilon}\right)}{\mu^{2}+\left(\Phi\left(n,\varphi_{\epsilon}\right)\right)^{2}}g_{n}X_{n}\right)\nonumber \\
 &  & +\sum_{n=1}^{\infty}\left(\frac{\Phi\left(n,\varphi_{\epsilon}\right)}{\mu^{2}+\left(\Phi\left(n,\varphi_{\epsilon}\right)\right)^{2}}g_{n}X_{n}-\frac{\Phi\left(n,\varphi_{\epsilon}\right)}{\mu^{2}+\left(\Phi\left(n,\varphi_{\epsilon}\right)\right)^{2}}g_{n}^{\epsilon}X_{n}\right)\\
 & = & A_{1}+A_{2}.\nonumber 
\end{eqnarray}
We consider two following estimates by diving into two steps.\\
\textbf{Step 1.} Estimate $\left\Vert A_{1}\right\Vert $
\begin{eqnarray}
A_{1} & \le & \sum_{n=1}^{\infty}\frac{\mu^{2}\Phi\left(n,\left|\mu-\mu_{\epsilon}\right|\right)+\Phi\left(n,\varphi\right)\Phi\left(n,\varphi_{\epsilon}\right)\Phi\left(n,\left|\varphi-\varphi_{\epsilon}\right|\right)}{\left[\mu^{2}+\left(\Phi\left(n,\varphi\right)\right)^{2}\right]\left[\mu^{2}+\left(\Phi\left(n,\varphi_{\epsilon}\right)\right)^{2}\right]}g_{n}X_{n}\nonumber \\
 & \le & \sum_{n=1}^{\infty}\frac{\left(\mu^{2}+\Phi\left(n,\varphi\right)\Phi\left(n,\varphi_{\epsilon}\right)\right)\Phi\left(n,\left|\varphi-\varphi_{\epsilon}\right|\right)}{4\mu^{2}\Phi\left(n,\varphi\right)\Phi\left(n,\varphi_{\epsilon}\right)}g_{n}X_{n}.\label{eq:33}
\end{eqnarray}
Notice that
\begin{equation}
\Phi\left(n,\varphi\right)\Phi\left(n,\varphi_{\epsilon}\right)\le B_{2}C_{2}\left(\int_{0}^{T}e^{n^{2}B\left(t\right)}dt\right)^{2},\quad\Phi\left(n,\left|\varphi-\varphi_{\epsilon}\right|\right)\le\left\Vert \varphi-\varphi_{\epsilon}\right\Vert _{L^{2}\left[0,T\right]}\left(\int_{0}^{T}e^{2n^{2}B\left(t\right)}dt\right)^{\frac{1}{2}}.
\end{equation}
It follows that
\begin{eqnarray}
\left\Vert A_{1}\right\Vert  & \le & \frac{\left\Vert \varphi-\varphi_{\epsilon}\right\Vert _{L^{2}\left[0,T\right]}\left(\mu^{2}+B_{2}C_{2}\left(\int_{0}^{T}e^{n^{2}B\left(t\right)}dt\right)^{2}\right)}{4\mu^{2}T^{2}B_{1}C_{1}}\left\Vert f\right\Vert \nonumber \\
 & \le & \frac{\left\Vert \varphi-\varphi_{\epsilon}\right\Vert _{L^{2}\left[0,T\right]}\left(\mu^{2}+n^{-4}D_{1}^{-2}B_{2}C_{2}\left(1-e^{-n^{2}D_{2}T}\right)^{2}\right)}{4\mu^{2}T^{2}B_{1}C_{1}}\left\Vert f\right\Vert \nonumber \\
 & \le & \frac{\left\Vert f\right\Vert \left(\mu^{2}+D_{1}^{-2}B_{2}C_{2}\right)}{4\mu^{2}T^{2}B_{1}C_{1}}\left\Vert \varphi-\varphi_{\epsilon}\right\Vert _{L^{2}\left[0,T\right]}.\label{eq:35}
\end{eqnarray}
\textbf{Step 2.} Estimate $\left\Vert A_{2}\right\Vert $\\
\begin{equation}
\left\Vert A_{2}\right\Vert \le\sqrt{\sum_{n=1}^{\infty}\frac{\left(\Phi\left(n,\varphi_{\epsilon}\right)\right)^{2}}{\mu^{2}+\left(\Phi\left(n,\varphi_{\epsilon}\right)\right)^{2}}\left|g_{n}-g_{n}^{\epsilon}\right|^{2}}\le\frac{1}{2\mu}\left\Vert g-g_{\epsilon}\right\Vert .\label{eq:36}
\end{equation}
Combining (\ref{eq:35}) and (\ref{eq:36}), the proof is completed.
\end{proof}
In order to obtain the boundedness of bias, we usually need some a
priori conditions. By Tikhonov's theorem, the \textcolor{red}{operator} $K^{-1}$ \textcolor{red}{is} restricted
to the continuous image of a compact set $M$. Thus, we assume $f$
is in a compact subset of $L^{2}\left(0,\pi\right)$. Hereafter,
we assume that $\left\Vert f\right\Vert _{H^{2k}\left(0,\pi\right)}\le M$
for $k>0$.
\begin{lemma}
\label{lem:4-1}If the a priori bound holds, then
\begin{equation}
\left\Vert f-f_{\mu}\right\Vert \le\begin{cases}
\max\left\{ 1,\frac{D_{2}^{2}}{\left[B_{1}\left(1-e^{-D_{1}T}\right)\right]^{2}}\right\} M\mu^{\frac{k}{2}} & ,0<k\le2,\\
\max\left\{ 1,\frac{D_{2}^{2}}{\left[B_{1}\left(1-e^{-D_{1}T}\right)\right]^{2}}\right\} M\mu & ,k>2.
\end{cases}
\end{equation}
\end{lemma}
\begin{proof}
From (\ref{eq:8}) and (\ref{eq:26}), we deduce that
\begin{eqnarray}
\left\Vert f-f_{\mu}\right\Vert ^{2} & \le & \sum_{n=1}^{\infty}\frac{\mu^{4}}{\left(\Phi\left(n,\varphi\right)\right)^{2}\left(\mu^{2}+\left(\Phi\left(n,\varphi\right)\right)^{2}\right)^{2}}g_{n}^{2}\nonumber \\
 & \le & \sum_{n=1}^{\infty}P\left(n\right)\frac{\left(1+n^{2}\right)^{2k}g_{n}^{2}}{\left(\Phi\left(n,\varphi\right)\right)^{2}},\label{eq:38}
\end{eqnarray}
where
\begin{equation}
P\left(n\right)=\frac{\mu^{4}}{\left(\mu^{2}+\left(\Phi\left(n,\varphi\right)\right)^{2}\right)^{2}}\left(1+n^{2}\right)^{-2k}.
\end{equation}
Next, we estimate $P\left(n\right)$. Without
loss of generality, we assume that $\mu^{-\frac{1}{4}}$ is not integer.
Therefore, (\ref{eq:38}) can be divided into the sum of $A_{1}$
and $A_{2}$ as follows:
\begin{equation}
A_{1}=\sum_{n=1}^{n_{0}}P\left(n\right)\frac{\left(1+n^{2}\right)^{2k}g_{n}^{2}}{\left(\Phi\left(n,\varphi\right)\right)^{2}},\quad A_{2}=\sum_{n=n_{0}+1}^{\infty}P\left(n\right)\frac{\left(1+n^{2}\right)^{2k}g_{n}^{2}}{\left(\Phi\left(n,\varphi\right)\right)^{2}},
\end{equation}
where $n_{0}\le\mu^{-\frac{1}{4}}\le n_{0}+1$.
In $A_{1}$, we have
\begin{equation}
P\left(n\right)\le\frac{\mu^{4}n^{8}D_{2}^{4}}{\left[B_{1}\left(1-e^{-D_{1}T}\right)\right]^{4}}\left(1+n^{2}\right)^{-2k}\le\frac{D_{2}^{4}}{\left[B_{1}\left(1-e^{-D_{1}T}\right)\right]^{4}}\mu^{4}n^{8-4k}.
\end{equation}
For $0<k\le2$, we deduce that
\begin{equation}
P\left(n\right)\le\frac{D_{2}^{4}}{\left[B_{1}\left(1-e^{-D_{1}T}\right)\right]^{4}}\mu^{k+2}\le\frac{D_{2}^{4}}{\left[B_{1}\left(1-e^{-D_{1}T}\right)\right]^{4}}\mu^{k}.\label{eq:42}
\end{equation}
For $k>2$, it yields
\begin{equation}
P\left(n\right)\le\frac{D_{2}^{4}}{\left[B_{1}\left(1-e^{-D_{1}T}\right)\right]^{4}}\mu^{4}.\label{eq:43}
\end{equation}
In addition, we observe in $A_{2}$ that
\begin{equation}
P\left(n\right)\le\left(1+n^{2}\right)^{-2k}\le\mu^{k}.\label{eq:44}
\end{equation}
From (\ref{eq:42})-(\ref{eq:44}), we thus obtain
\begin{equation}
P\left(n\right)\le\begin{cases}
\max\left\{ 1,\frac{D_{2}^{4}}{\left[B_{1}\left(1-e^{-D_{1}T}\right)\right]^{4}}\right\} \mu^{k} & ,0<k\le2,\\
\max\left\{ 1,\frac{D_{2}^{4}}{\left[B_{1}\left(1-e^{-D_{1}T}\right)\right]^{4}}\right\} \mu^{2} & ,k>2.
\end{cases}
\end{equation}
Hence, by using the assumption, we conclude that
\begin{equation}
\left\Vert f-f_{\mu}\right\Vert \le\begin{cases}
\max\left\{ 1,\frac{D_{2}^{2}}{\left[B_{1}\left(1-e^{-D_{1}T}\right)\right]^{2}}\right\} M\mu^{\frac{k}{2}} & ,0<k\le2,\\
\max\left\{ 1,\frac{D_{2}^{2}}{\left[B_{1}\left(1-e^{-D_{1}T}\right)\right]^{2}}\right\} M\mu & ,k>2.
\end{cases}
\end{equation}
\end{proof}
\begin{theorem}
Assume that the a priori condition and the noise assumption hold,
the following estimates are obtained.
\begin{description}
\item [{a.}] If $0<k\le2$ and choose $\mu=\left(\dfrac{\epsilon}{M}\right)^{\frac{1}{k+2}}$,
then
\end{description}
\begin{equation}
\left\Vert f-f_{\mu}^{\epsilon}\right\Vert \le P\left[\left\Vert f\right\Vert \left(\epsilon^{\frac{2}{k+2}}+D_{1}^{-2}B_{2}C_{2}M^{\frac{2}{k+2}}\right)\epsilon^{\frac{k}{k+2}}+\epsilon^{\frac{k+1}{k+2}}+\epsilon^{\frac{1}{k+2}}\right],
\end{equation}
where $P$ is constant and depends on constants $T,B_{1},C_{1},D_{1},D_{2}$, and  $M$ (shown in Sec. 2). 
\begin{description}
\item [{b.}] If $k>2$ and choose $\mu=\left(\dfrac{\epsilon}{M}\right)^{\frac{1}{2}}$,
then
\end{description}
\begin{equation}
\left\Vert f-f_{\mu}^{\epsilon}\right\Vert \le Q\epsilon\left[\left\Vert f\right\Vert \left(1+\frac{D_{1}^{-2}B_{2}C_{2}}{M}\epsilon\right)+1\right],
\end{equation}
where $Q$ is constant and depends on $T,B_{1},C_{1},D_{1},D_{2},M$. 
\end{theorem}
\begin{proof}
From Lemma \ref{lem:3} and Lemma \ref{lem:4-1}, we can obtain the
proof easily. Indeed, for $0<k\le2$, using $\mu=\left(\dfrac{\epsilon}{M}\right)^{\frac{1}{k+2}}$
we have
\begin{eqnarray}
\left\Vert f-f_{\mu}^{\epsilon}\right\Vert  & \le & \frac{\left\Vert f\right\Vert }{4T^{2}B_{1}C_{1}}\left(1+\frac{D_{1}^{-2}B_{2}C_{2}}{\mu^{2}}\right)\left\Vert \varphi-\varphi_{\epsilon}\right\Vert _{L^{2}\left[0,T\right]}+\frac{1}{2\mu}\left\Vert g-g_{\epsilon}\right\Vert \nonumber \\
 &  & +\max\left\{ 1,\frac{D_{2}^{2}}{\left[B_{1}\left(1-e^{-D_{1}T}\right)\right]^{2}}\right\} M\mu^{\frac{k}{2}}\nonumber \\
 & \le & \frac{\left\Vert f\right\Vert }{4T^{2}B_{1}C_{1}}\left(\epsilon^{\frac{2}{k+2}}+D_{1}^{-2}B_{2}C_{2}M^{\frac{2}{k+2}}\right)\epsilon^{\frac{k}{k+2}}+\frac{1}{2}M^{\frac{1}{k+2}}\epsilon^{\frac{k+1}{k+2}}\nonumber \\
 &  & +\max\left\{ 1,\frac{D_{2}^{2}}{\left[B_{1}\left(1-e^{-D_{1}T}\right)\right]^{2}}\right\} M\epsilon^{\frac{1}{k+2}}.
\end{eqnarray}
Besides, for $k>2$, choosing $\mu=\left(\dfrac{\epsilon}{M}\right)^{\frac{1}{2}}$
will lead to the following estimate. 
\begin{eqnarray}
\left\Vert f-f_{\mu}^{\epsilon}\right\Vert  & \le & \frac{\left\Vert f\right\Vert }{4T^{2}B_{1}C_{1}}\left(1+\frac{D_{1}^{-2}B_{2}C_{2}}{M}\epsilon\right)\epsilon+\frac{1}{2}M^{\frac{1}{2}}\epsilon^{\frac{1}{2}}\nonumber \\
 &  & +\max\left\{ 1,\frac{D_{2}^{2}}{\left[B_{1}\left(1-e^{-D_{1}T}\right)\right]^{2}}\right\} M^{\frac{1}{2}}\epsilon^{\frac{1}{2}}.
\end{eqnarray}
\end{proof}

\section{Tikhonov regularization under a posteriori parameter choice rule}

In this section, we consider an a posteriori regularization
parameter choice in Morozov's discrepancy principle (see in \cite{Sc,CP}). First, we introduce \textcolor{red}{the} following lemma:
\begin{lemma}
\label{lem:4}Set $\rho\left(\mu\right)=\left\Vert Kf_{\mu}^{\epsilon}-g_{\epsilon}\right\Vert $
and assume that $0<\epsilon<\left\Vert g_{\epsilon}\right\Vert $,
then the following results hold:
\begin{description}
\item [{a.}] $\rho\left(\mu\right)$ is a continuous function.
\item [{b.}] $\rho\left(\mu\right)\to0$ as $\mu\to0$.
\item [{c.}] $\rho\left(\mu\right)\to\left\Vert g_{\epsilon}\right\Vert $
as $\mu\to\infty$.
\item [{d.}] $\rho\left(\mu\right)$ is a strictly increasing function.
\end{description}
\end{lemma}
\begin{proof}
All results are derived from
\begin{equation}
\rho\left(\mu\right)=\sqrt{\sum_{n=1}^{\infty}\left(\frac{\mu^{2}}{\mu^{2}+\left(\Phi\left(n,\varphi\right)\right)^{2}}\right)^{2}\left(g_{n}^{\epsilon}\right)^{2}}.
\end{equation}
\end{proof}
Let us define a function $H\left(y\right)$ as follows.
\begin{equation}
H\left(y\right)=\begin{cases}
y^{y}\left(1-y\right)^{1-y} & ,y\in\left(0,1\right),\\
1 & ,y=\left\{ 0;1\right\} .
\end{cases}
\end{equation}
It is clear that $0<H\left(y\right)\le1$ since we have
\begin{equation}
\sup_{x>0}\frac{x^{y}}{1+x}=H\left(y\right),\quad y\in\left[0,1\right].
\end{equation}
\begin{lemma}
Choose $\tau>1$ such that $0<\tau\epsilon<\left\Vert g_{\epsilon}\right\Vert $,
then there exists a unique regularization parameter $\mu>0$ such
that $\left\Vert Kf_{\mu}^{\epsilon}-g_{\epsilon}\right\Vert =\tau\epsilon$.
Moreover, if the a priori condition with $k\in\left(0,1\right]$ and
the noise assumptions hold, we have the following inequality
\begin{equation}
\frac{\epsilon}{\mu^{k+1}}\le\frac{P}{\tau-1}H\left(\frac{1-k}{2}\right)M,
\end{equation}
where $P$ is constant, and depends on the constants $k,T,B_{2},C_{1},D_{1},D_{2}$.\end{lemma}
\begin{proof}
The uniqueness of regularization parameter $\mu>0$ is derived from
Lemma \ref{lem:4}. We thus only need to prove the inequality.
First, we notice that
\begin{eqnarray}
\tau\epsilon & = & \sqrt{\sum_{n=1}^{\infty}\left(\frac{\mu^{2}}{\mu^{2}+\left(\Phi\left(n,\varphi_{\epsilon}\right)\right)^{2}}\right)^{2}\left(g_{n}^{\epsilon}\right)^{2}}\nonumber \\
 & \le & \sqrt{\sum_{n=1}^{\infty}\left(\frac{\mu^{2}}{\mu^{2}+\left(\Phi\left(n,\varphi_{\epsilon}\right)\right)^{2}}\right)^{2}\left(g_{n}-g_{n}^{\epsilon}\right)^{2}}\nonumber \\
 &  & +\sqrt{\sum_{n=1}^{\infty}\left(\frac{\mu^{2}\Phi\left(n,\varphi\right)}{\left(\mu^{2}+\left(\Phi\left(n,\varphi_{\epsilon}\right)\right)^{2}\right)\left(1+n^{2}\right)^{k}}\right)^{2}\frac{\left(1+n^{2}\right)^{2k}\left(g_{n}\right)^{2}}{\left(\Phi\left(n,\varphi\right)\right)^{2}}}.
\end{eqnarray}
Due to ${\displaystyle \frac{\mu^{2}}{\mu^{2}+\left(\Phi\left(n,\varphi_{\epsilon}\right)\right)^{2}}\le1}$
and setting
\begin{equation}
K\left(n\right)=\frac{\mu^{2}\Phi\left(n,\varphi\right)}{\left(\mu^{2}+\left(\Phi\left(n,\varphi_{\epsilon}\right)\right)^{2}\right)\left(1+n^{2}\right)^{k}},
\end{equation}
we then have
\begin{equation}
\tau\epsilon\le\epsilon+\sqrt{\sum_{n=1}^{\infty}K^{2}\left(n\right)\frac{\left(1+n^{2}\right)^{2k}\left(g_{n}\right)^{2}}{\left(\Phi\left(n,\varphi\right)\right)^{2}}}.\label{eq:57}
\end{equation}
Now, we estimate $K\left(n\right)$ as follows.
\begin{eqnarray}
K\left(n\right) & \le & \frac{\left(\frac{\mu}{\Phi\left(n,\varphi_{\epsilon}\right)}\right)^{1-k}\Phi\left(n,\varphi\right)\left(\Phi\left(n,\varphi_{\epsilon}\right)\right)^{k-3}}{\left(\left(\frac{\mu}{\Phi\left(n,\varphi_{\epsilon}\right)}\right)^{2}+1\right)\left(1+n^{2}\right)^{k}}\mu^{k+1}\nonumber \\
 & \le & H\left(\frac{1-k}{2}\right)\mu^{k+1}B_{2}C_{1}D_{1}^{-1}D_{2}^{3-k}\frac{\left(1-e^{-D_{1}T}\right)^{k-3}}{n^{2\left(k-2\right)}\left(1+n^{2}\right)^{k}}\nonumber \\
 & \le & H\left(\frac{1-k}{2}\right)\mu^{k+1}B_{2}C_{1}D_{1}^{-1}D_{2}^{3-k}\left(1-e^{-D_{1}T}\right)^{k-3}.\label{eq:58}
\end{eqnarray}
Therefore, combining (\ref{eq:57}) and (\ref{eq:58}), we conclude
that
\begin{equation}
\tau\epsilon\le\epsilon+B_{2}C_{1}D_{1}^{-1}D_{2}^{3-k}\left(1-e^{-D_{1}T}\right)^{k-3}H\left(\frac{1-k}{2}\right)\mu^{k+1}M,\label{eq:59}
\end{equation}
which gives the desired result.\end{proof}
\begin{theorem}
\label{thm:8}Assume the a priori condition and the noise assumptions
hold, and there exists $\tau>1$ such that $0<\tau\epsilon<\left\Vert g_{\epsilon}\right\Vert $.
Then, we choose a unique regularization parameter $\mu>0$ such that
\begin{equation}
\left\Vert f-f_{\mu}^{\epsilon}\right\Vert \le\begin{cases}
\epsilon^{\frac{k}{k+1}}P & ,0<k\le1,\\
\epsilon^{\frac{1}{2}}Q & ,k>1,
\end{cases}
\end{equation}
where constants $P$ and $Q$ depend on the constants $T,\mu,k,\tau,B_{1},B_{2},C_{1},C_{2},D_{1},D_{2}$
and $M$.\end{theorem}
\begin{proof}
For $0<k\le1$, we have
\begin{eqnarray}
\left\Vert f-f_{\mu}^{\epsilon}\right\Vert  & = & \sum_{n=1}^{\infty}\frac{1}{\Phi\left(n,\varphi\right)}\left[g_{n}-\frac{\left(\Phi\left(n,\varphi\right)\right)^{2}}{\mu^{2}+\left(\Phi\left(n,\varphi\right)\right)^{2}}g_{n}^{\epsilon}\right]\nonumber \\
 &  & +\sum_{n=1}^{\infty}\left[\frac{\Phi\left(n,\varphi\right)}{\mu^{2}+\left(\Phi\left(n,\varphi\right)\right)^{2}}-\frac{\Phi\left(n,\varphi_{\epsilon}\right)}{\mu^{2}+\left(\Phi\left(n,\varphi_{\epsilon}\right)\right)^{2}}\right]g_{n}^{\epsilon}.
\end{eqnarray}
It follows that
\begin{eqnarray}
\left\Vert f-f_{\mu}^{\epsilon}\right\Vert ^{2} & \le & 2\sum_{n=1}^{\infty}\frac{1}{\left(\Phi\left(n,\varphi\right)\right)^{2}}\left[g_{n}-\frac{\left(\Phi\left(n,\varphi\right)\right)^{2}}{\mu^{2}+\left(\Phi\left(n,\varphi\right)\right)^{2}}g_{n}^{\epsilon}\right]^{2}\nonumber \\
 &  & +2\sum_{n=1}^{\infty}\left(\left[\frac{\Phi\left(n,\varphi\right)}{\mu^{2}+\left(\Phi\left(n,\varphi\right)\right)^{2}}-\frac{\Phi\left(n,\varphi_{\epsilon}\right)}{\mu^{2}+\left(\Phi\left(n,\varphi_{\epsilon}\right)\right)^{2}}\right]g_{n}^{\epsilon}\right)^{2}.\label{eq:62}
\end{eqnarray}
We set $K_{1}$ and $K_{2}$ as follows.
\begin{equation}
K_{1}\left(n\right)=\frac{1}{\left(\Phi\left(n,\varphi\right)\right)^{2}}\left[g_{n}-\frac{\left(\Phi\left(n,\varphi\right)\right)^{2}}{\mu^{2}+\left(\Phi\left(n,\varphi\right)\right)^{2}}g_{n}^{\epsilon}\right]^{\frac{2}{k+1}}\left[g_{n}-\frac{\left(\Phi\left(n,\varphi\right)\right)^{2}}{\mu^{2}+\left(\Phi\left(n,\varphi\right)\right)^{2}}g_{n}^{\epsilon}\right]^{\frac{2k}{k+1}},
\end{equation}
\begin{equation}
K_{2}\left(n\right)=\left(\left[\frac{\Phi\left(n,\varphi\right)}{\mu^{2}+\left(\Phi\left(n,\varphi\right)\right)^{2}}-\frac{\Phi\left(n,\varphi_{\epsilon}\right)}{\mu^{2}+\left(\Phi\left(n,\varphi_{\epsilon}\right)\right)^{2}}\right]g_{n}^{\epsilon}\right)^{2}.
\end{equation}
Afterwards, we estimate $\left\Vert f-f_{\mu}^{\epsilon}\right\Vert $
by considering the following inequalities. First, we see that
\begin{equation}
\sum_{n=1}^{\infty}K_{1}\left(n\right)\le L_{1}^{\frac{k}{k+1}}L_{2}^{\frac{1}{k+1}},
\end{equation}
where
\begin{equation}
L_{1}=\sum_{n=1}^{\infty}\left(g_{n}-\frac{\left(\Phi\left(n,\varphi\right)\right)^{2}}{\mu^{2}+\left(\Phi\left(n,\varphi\right)\right)^{2}}g_{n}^{\epsilon}\right)^{2},\label{eq:66}
\end{equation}
\begin{equation}
L_{2}=\sum_{n=1}^{\infty}\left(\Phi\left(n,\varphi\right)\right)^{-2\left(k+1\right)}\left(g_{n}-\frac{\left(\Phi\left(n,\varphi\right)\right)^{2}}{\mu^{2}+\left(\Phi\left(n,\varphi\right)\right)^{2}}g_{n}^{\epsilon}\right)^{2}.
\end{equation}
Now, to estimate $B_{1}$ and $B_{2}$ as follows:
\begin{equation}
L_{1}^{\frac{1}{2}}\le\sqrt{\sum_{n=1}^{\infty}\left(g_{n}-g_{n}^{\epsilon}\right)^{2}}+\sqrt{\sum_{n=1}^{\infty}\left(g_{n}^{\epsilon}-\frac{\left(\Phi\left(n,\varphi\right)\right)^{2}}{\mu^{2}+\left(\Phi\left(n,\varphi\right)\right)^{2}}g_{n}^{\epsilon}\right)^{2}}\le\left(1+\tau\right)\epsilon.
\end{equation}
\begin{equation}
L_{2}^{\frac{1}{2}}\le\sqrt{\sum_{n=1}^{\infty}\left(\Phi\left(n,\varphi\right)\right)^{-2\left(k+1\right)}g_{n}^{2}}+\sqrt{\sum_{n=1}^{\infty}\left(\Phi\left(n,\varphi\right)\right)^{-2\left(k+1\right)}\left(\frac{\left(\Phi\left(n,\varphi\right)\right)^{2}}{\mu^{2}+\left(\Phi\left(n,\varphi\right)\right)^{2}}\right)^{2}\left(g_{n}^{\epsilon}\right)^{2}}.\label{eq:69}
\end{equation}
In (\ref{eq:69}), we continue to estimate two terms (denoted by $L_{3}$
and $L_{4}$, respectively) by direct computation.
\begin{eqnarray}
L_{3} & \le & \sqrt{\sum_{n=1}^{\infty}\left(\Phi\left(n,\varphi\right)\left(1+n^{2}\right)\right)^{-2k}\frac{\left(1+n^{2}\right)^{2k}g_{n}^{2}}{\left(\Phi\left(n,\varphi\right)\right)^{2}}}\nonumber \\
 & \le & \sqrt{\sum_{n=1}^{\infty}\left(\frac{n^{2}D_{2}}{B_{1}\left(1-e^{-D_{1}T}\right)\left(1+n^{2}\right)}\right)^{2k}\frac{\left(1+n^{2}\right)^{2k}g_{n}^{2}}{\left(\Phi\left(n,\varphi\right)\right)^{2}}}\nonumber \\
 & \le & \left(\frac{D_{2}}{B_{1}\left(1-e^{-D_{1}T}\right)}\right)^{k}M.\label{eq:70}
\end{eqnarray}
\begin{eqnarray}
L_{4} & \le & \sqrt{\sum_{n=1}^{\infty}\left(\Phi\left(n,\varphi\right)\right)^{-2\left(k+1\right)}\left(\frac{\left(\Phi\left(n,\varphi\right)\right)^{2}}{\mu^{2}+\left(\Phi\left(n,\varphi\right)\right)^{2}}\right)^{2}\left(g_{n}-g_{n}^{\epsilon}\right)^{2}}\nonumber \\
 &  & +\sqrt{\sum_{n=1}^{\infty}\left(\Phi\left(n,\varphi\right)\right)^{-2\left(k+1\right)}\left(\frac{\left(\Phi\left(n,\varphi\right)\right)^{2}}{\mu^{2}+\left(\Phi\left(n,\varphi\right)\right)^{2}}\right)^{2}\left(g_{n}\right)^{2}}.\label{eq:71}
\end{eqnarray}
In (\ref{eq:71}), we denote two terms in the right-hand side by $L_{5}$
and $L_{6}$. Then, using Theorem \ref{thm:8} and (\ref{eq:70})
we estimate them as follows.
\begin{eqnarray}
L_{5} & \le & \sqrt{\sum_{n=1}^{\infty}\mu^{-2\left(k+1\right)}\left(\frac{\left(\frac{\mu}{\Phi\left(n,\varphi\right)}\right)^{k+1}}{\left(\frac{\mu}{\Phi\left(n,\varphi\right)}\right)^{2}+1}\right)^{2}\left(g_{n}-g_{n}^{\epsilon}\right)^{2}}\nonumber \\
 & \le & \mu^{-\left(k+1\right)}H\left(\frac{k+1}{2}\right)\epsilon\nonumber \\
 & \le & \frac{P}{\tau-1}H\left(\frac{1-k}{2}\right)H\left(\frac{k+1}{2}\right)M\nonumber \\
 & \le & \frac{P}{\tau-1}\left(H\left(\frac{1-k}{2}\right)\right)^{2}M.\label{eq:72}
\end{eqnarray}
\begin{equation}
L_{6}\le\sqrt{\sum_{n=1}^{\infty}\left(\Phi\left(n,\varphi\right)\right)^{-2\left(k+1\right)}\left(g_{n}\right)^{2}}\le\left(\frac{D_{2}}{B_{1}\left(1-e^{-D_{1}T}\right)}\right)^{k}M.\label{eq:73}
\end{equation}
Therefore, from (\ref{eq:66})-(\ref{eq:73}) we have
\begin{equation}
\sum_{n=1}^{\infty}K_{1}\left(n\right)\le\left[\left(1+\tau\right)\epsilon\right]^{\frac{2k}{k+1}}\left[2\left(\frac{D_{2}}{B_{1}\left(1-e^{-D_{1}T}\right)}\right)^{k}M+\frac{P}{\tau-1}\left(H\left(\frac{1-k}{2}\right)\right)^{2}M\right]^{\frac{2}{k+1}}.\label{eq:74}
\end{equation}
Next, estimating $K_{2}\left(n\right)$ is in process. Similarly,
we look back (\ref{eq:33})-(\ref{eq:35}), then see that
\begin{eqnarray}
\sqrt{\sum_{n=1}^{\infty}K_{2}\left(n\right)} & \le & \sqrt{\sum_{n=1}^{\infty}\left(\left[\frac{\Phi\left(n,\varphi\right)}{\mu^{2}+\left(\Phi\left(n,\varphi\right)\right)^{2}}-\frac{\Phi\left(n,\varphi_{\epsilon}\right)}{\mu^{2}+\left(\Phi\left(n,\varphi_{\epsilon}\right)\right)^{2}}\right]\left(g_{n}^{\epsilon}-g_{n}\right)\right)^{2}}\nonumber \\
 &  & +\sqrt{\sum_{n=1}^{\infty}\left(\left[\frac{\Phi\left(n,\varphi\right)}{\mu^{2}+\left(\Phi\left(n,\varphi\right)\right)^{2}}-\frac{\Phi\left(n,\varphi_{\epsilon}\right)}{\mu^{2}+\left(\Phi\left(n,\varphi_{\epsilon}\right)\right)^{2}}\right]g_{n}\right)^{2}}\nonumber \\
 & \le & \left\Vert \varphi_{\epsilon}-\varphi\right\Vert _{L^{2}\left[0,T\right]}\left(P\left\Vert g_{\epsilon}-g\right\Vert +\left\Vert f\right\Vert \right).\label{eq:75}
\end{eqnarray}
Combining (\ref{eq:62}), (\ref{eq:74}) and (\ref{eq:75}) gives
the first desired. Moreover, we can obtain the second by embedding
$H^{2k}$ into $H^{1}$.
\end{proof}
\section{Numerical Examples}
In this section, we \textcolor{red}{shall implement our} proposed regularization methods.
Two different numerical examples \textcolor{red}{(choosing $k=T=1$)} are shown. The first example is to consider an example with $a$ in Eq. (2) is a constant, and the function
$f$ obtained from exact data function. The second example is to
consider an example with $a$ is a non-constant function, and $f$
obtained from observation data of $g$ and $\varphi$. \\
The couple of $\left(g_{\epsilon},\varphi_{\epsilon}\right)$, which \textcolor{red}{is} determined below, plays as measured data with a random noise as follows:
\begin{eqnarray}
g_{\epsilon}\left(.\right) & = & g\left(.\right)\left(1+\frac{\epsilon\cdot\mbox{rand}\left(.\right)}{\left\Vert g\right\Vert }\right),\\
\varphi_{\epsilon}\left(.\right) & = & \varphi\left(.\right)+\epsilon\cdot\mbox{rand}\left(.\right),
\end{eqnarray}
where $\mathtt{rand(\cdot)} \in (-1,1)$ is a random number. \textcolor{red}{By those measured ways}, we can easily verify the validity of the inequality:
$$
\| g - g_{\epsilon}  \| \leq \epsilon
$$
and
$$
\|\varphi  - \varphi_{\epsilon} \| \leq \epsilon
$$
In addition, \textcolor{red}{there is a possibility of taking} the regularization parameter for the a priori
parameter choice rule $\mu=\left(\dfrac{\epsilon}{M}\right)^{\frac{1}{3}}$
where $M$ plays a role as a priori condition computed by $\left\Vert f\right\Vert _{H^{2}\left(0,\pi\right)}$.
The absolute and relative errors between regularized and exact solutions are \textcolor{red}{computed}. 
On the other hand, the regularized solution are \textcolor{red}{defined} by:
\begin{equation}
f_{\mu}^{\epsilon}\left(x\right)=\dfrac{2}{\pi}\sum_{n=1}^{N}\frac{\Phi\left(n,\varphi_{\epsilon}\right)}{\mu^{2}+\left(\Phi\left(n,\varphi_{\epsilon}\right)\right)^{2}}\left\langle g_{\epsilon}\left(x\right),\sin\left(nx\right)\right\rangle \sin\left(nx\right),
\end{equation}
\begin{equation}
\Phi\left(n,\varphi_{\epsilon}\right)=\int_{0}^{1}e^{n^{2}\left(A\left(t\right)-A\left(1\right)\right)}\varphi_{\epsilon}\left(t\right)dt,
\end{equation}
where $N$ is the truncation number; whereby $N =1000$ is chosen in the examples.\\
In general, the whole numerical procedure is shown in the following steps:\\
\textbf{Step 1.} Choosing $L$ and $K$ to generate temporal and spatial discretizations as follows:
\begin{equation}
x_{j}=j\Delta x,\Delta x=\frac{\pi}{K},j=\overline{0,K},
\end{equation}
\begin{equation}
t_{i}=i\Delta t,\Delta t=\frac{1}{L},i=\overline{0,L}.
\end{equation}
\textcolor{red}{Obviously}, the higher value of $L$ and $K$ will provide more stable and accurate numerical calculation, however in \textcolor{red}{our} examples $L=K=100$ are satisfied.\\
\textbf{Step 2.} Setting $f_{\mu}^{\epsilon}\left(x_{j}\right)=f_{\mu,j}^{\epsilon}$
and $f\left(x_{j}\right)=f_{j}$, constructing two vectors \textcolor{red}{containing} all
discrete values of $f_{\mu}^{\epsilon}$ and $f$ denoted by $\Lambda_{\mu}^{\epsilon}$
and $\Psi$, respectively.
\begin{equation}
\Lambda_{\mu}^{\epsilon}=\begin{bmatrix}f_{\mu,0}^{\epsilon} & f_{\mu,1}^{\epsilon} & ... & f_{\mu,K-1}^{\epsilon} & f_{\mu,K}^{\epsilon}\end{bmatrix}\in\mathbb{R}^{K+1},
\end{equation}
\begin{equation}
\Psi=\begin{bmatrix}f_{0} & f_{1} & ... & f_{K-1} & f_{K}\end{bmatrix}\in\mathbb{R}^{K+1}.
\end{equation}
\textbf{Step 3.} Error estimate between the exact solution and regularized solution;\\
Absolute error estimation:
\begin{eqnarray}
E_{1} & = & \sqrt{\frac{1}{K+1}\sum_{j=0}^{K}\left|f_{\mu}^{\epsilon}\left(x_{j}\right)-f\left(x_{j}\right)\right|^{2}},\\
\end{eqnarray}
Relative error estimation:
\begin{eqnarray}
E_{2} & = & \frac{\sqrt{\sum_{j=0}^{K}\left|f_{\mu}^{\epsilon}\left(x_{j}\right)-f\left(x_{j}\right)\right|^{2}}}{\sqrt{\sum_{j=0}^{K}\left|f\left(x_{j}\right)\right|^{2}}}.
\end{eqnarray}
\subsection{Example 1}
As mentioned above, in this example we consider $a$ is constant, and $f$ is an exact data function. Specifically, we consider a type of the problem (\ref{eq:2})-(\ref{eq:4}) as follows
\begin{equation}
\begin{cases}
u_{t}-u_{xx}=2^{-1}\left(e^{t}-1\right)\sin2x; & \left(x,t\right)\in\left(0,\pi\right)\times\left(0,1\right),\\
u\left(x,0\right)=0,\quad u\left(x,1\right)=10^{-1}\left(e-1\right)\sin2x; & x\in\left[0,\pi\right],\\
u\left(0,t\right)=u\left(\pi,t\right)=0; & t\in\left[0,1\right],
\end{cases}
\end{equation}
This implies that $a\left(t\right)=1,\varphi\left(t\right)=e^{t}-1,g\left(x\right)=10^{-1}\left(e-1\right)\sin2x$
and $f\left(x\right)=2^{-1}\sin2x$. \\
It is easy to see that $u\left(x,t\right)=10^{-1}\left(e^{t}-1\right)\sin2x$ is the unique solution of the problem.
Next, we establish the regularized
solution according to composite Simpson's rule.
\begin{equation}
f_{\mu}^{\epsilon}\left(x\right)=\frac{e-1}{10}\left(1+\frac{\epsilon\cdot\mbox{rand}\left(.\right)}{\sqrt{\pi}\left\Vert g\right\Vert }\right)\frac{\Phi\left(2,\varphi_{\epsilon}\right)}{\mu^{2}+\left(\Phi\left(2,\varphi_{\epsilon}\right)\right)^{2}}\sin2x,
\end{equation}
\begin{equation}
\Phi\left(2,\varphi_{\epsilon}\right)=\frac{1}{3M}\left[h\left(t_{0}\right)+2\sum_{i=1}^{\frac{L}{2}-1}h\left(t_{2i}\right)+4\sum_{i=1}^{\frac{L}{2}}h\left(t_{2i-1}\right)+h\left(t_{M}\right)\right],
\end{equation}
\begin{equation}
h\left(t_{i}\right)=e^{4\left(t_{i}-1\right)}\left(\varphi\left(t_{i}\right)+\epsilon\cdot\left|\mbox{rand}\left(t_{i}\right)\right|\right).
\end{equation}
In practice, it is very difficult to obtain the value of $M$ without having
an exact solution. We thus try taking $M=1000$ leading to $\mu_{1}=\dfrac{\epsilon^{\frac{1}{3}}}{10}$
for the a priori parameter choice rule, and $\mu_{2}=\dfrac{\epsilon^{\frac{9}{20}}}{40}$
for the a posteriori parameter choice rule based on (\ref{eq:59})
with $\tau=1.5$.
\subsection{Example 2}
Similar to the first example, however in this example we consider $a\left(t\right)=2t+1$, then $A\left(t\right)-A\left(1\right)=t^{2}+t-2$; we choose
\begin{equation}
g\left(x\right)=e^{3}\sum_{m=1}^{3}\sin\left(mx\right),\quad\varphi\left(t\right)=1.
\end{equation}
Thus, the exact solution is obtained by:
\begin{equation}
f\left(x\right)=\dfrac{2e^{3}}{\pi}\sum_{n=1}^{1000}\frac{1}{\Phi\left(n,1\right)}\sum_{m=1}^{3}\left\langle \sin\left(mx\right),\sin\left(nx\right)\right\rangle \sin\left(nx\right),\label{eq:92}
\end{equation}
\begin{equation}
\Phi\left(n,1\right)=\frac{1}{n^{2}}\left(1-e^{-2n^{2}}\right).\label{eq:93}
\end{equation}
Unlike the first example, from the analytical
solution, we can have $\left\Vert f\right\Vert _{H^{2}\left(0,\pi\right)}<5500$
which implies that $\mu_{1}=\left(\dfrac{\epsilon}{5500}\right)^{\frac{1}{3}}$
for the a priori parameter choice rule. Afterwards, based on (\ref{eq:59})
with $\tau=1.1$ again, we can compute the regularization parameter
for the a posteriori parameter choice rule, $\mu_{2}=\dfrac{\epsilon^{\frac{1}{2}}}{1100}$.
Therefore, the regularized solution can be computed by
\begin{equation}
f_{\mu}^{\epsilon}\left(x\right)=e^{3}\left(1+\frac{\epsilon\cdot\mbox{rand}\left(.\right)}{\sqrt{\pi}\left\Vert g\right\Vert }\right)\sum_{n=1}^{3}\frac{n^{2}\left(1-e^{-2n^{2}}\right)\left(1+\epsilon\cdot\mbox{rand}\left(.\right)\right)}{n^{4}\mu^{2}+\left(\left(1+\epsilon\cdot\mbox{rand}\left(.\right)\right)\left(1-e^{-2n^{2}}\right)\right)^{2}}\sin\left(nx\right).
\end{equation}
\\
Tables \ref{tab:2} and \ref{tab:1} show the absolute and relative error estimates between the exact solution and its regularized solution for both, the a priori and the a posteriori, parameter choice rules in the numerical examples. In the first example as shown in Table 1; when $a$ is constant, and $f$ is an exact data function; it shows the convergence speed of both parameter choice rule methods are quite similar and slow as $\epsilon$ tends to $0$. Whereas, in the second example shown in Table 2; when $a$ is not a constant, and $f$ is obtained from measured data; it shows the convergence speed of the a posteriori parameter choice rule is better than (by second order) the a priori parameter choice rule as $\epsilon$ tends to $0$.

In addition, Figures 1 and 2 show a comparison between the exact solution and its regularized solution for the a priori parameter and the a posteriori parameter choice rules in the first example, respectively. It again shows that for the both parameter choice rule methods, the regularized solution was strong oscillated around the exact solution when $\epsilon$ around $0.1$; nevertheless it converges to the exact solution as $\epsilon$ tends to $0$. In the second example, Figures 3 and 4 show the same tendency as in the first example for both methods. 

\section{Conclusion}

In this study, we solved the problem (\ref{eq:2})-(\ref{eq:4}) to recover temperature function of the unknown sources in the parabolic equation with the time-dependent
coefficient (i.e. inhomogeneous source) by suggesting two methods, the
a priori and a posteriori parameter choice rules. 

In the theoretical results, we obtained the error estimates of both methods based on a priori condition. From the numerical results, it shows that the regularized solutions are converged to the exact solutions. Furthermore, it also shows that the a posteriori parameter choice rule method is better than the a priori parameter choice rule method in term of the convergence speed.

\section*{Acknowledgments} This research is funded by Foundation for Science and Technology
Development of Ton Duc Thang University (FOSTECT) under project FOSTECT.2014.BR.03.

%\end{document}

\newpage

\begin{table}
\noindent \begin{centering}
\begin{tabular}{|c|c|c|c|c|}
\hline 
$\epsilon$ & $E_{1}^{\mu_{1}}$ & $E_{2}^{\mu_{1}}$ & $E_{1}^{\mu_{2}}$ & $E_{2}^{\mu_{2}}$\tabularnewline
\hline 
5.0E-01 & 2.19973047E-01 & 6.25280883E-01 & 2.37722314E-01 & 6.75733186E-01\tabularnewline
\hline 
1.0E-01 & 4.84352308E-02 & 1.37678794E-01 & 5.74458438E-02 & 1.63291769E-01\tabularnewline
\hline 
5.0E-02 & 2.61788488E-02 & 7.44142701E-02 & 2.83719452E-02 & 8.06482211E-02\tabularnewline
\hline 
1.0E-02 & 8.18020277E-03 & 2.32525051E-02 & 8.06299020E-03 & 2.29193244E-02\tabularnewline
\hline 
5.0E-03 & 5.77361943E-03 & 1.64117100E-02 & 5.73993055E-03 & 1.63159482E-02\tabularnewline
\hline 
1.0E-03 & 4.80971954E-03 & 1.36717917E-02 & 4.54649392E-03 & 1.29235639E-02\tabularnewline
\hline 
5.0E-04 & 4.67487514E-03 & 1.32884919E-02 & 4.54039711E-03 & 1.29062335E-02\tabularnewline
\hline 
1.0E-04 & 4.48497982E-03 & 1.27487080E-02 & 4.41615901E-03 & 1.25530824E-02\tabularnewline
\hline 
5.0E-05 & 4.44987772E-03 & 1.26489290E-02 & 4.41003343E-03 & 1.25356703E-02\tabularnewline
\hline 
1.0E-05 & 4.41942400E-03 & 1.25623633E-02 & 4.40598999E-03 & 1.25241767E-02\tabularnewline
\hline 
\end{tabular}
\par\end{centering}

\caption{Error estimate between the exact solution and its regularized solution for the
a priori parameter and the a posteriori parameter choice rules in Example 1.\label{tab:2}}
\end{table}

\begin{table}
\noindent \begin{centering}
\begin{tabular}{|c|c|c|c|c|}
\hline 
$\epsilon$ & $E_{1}^{\mu_{1}}$ & $E_{2}^{\mu_{1}}$ & $E_{1}^{\mu_{2}}$ & $E_{2}^{\mu_{2}}$\tabularnewline
\hline 
5.0E-01 & 8.73911561E+01 & 6.23557272E-01 & 1.68095756E+02 & 1.19940433E+00\tabularnewline
\hline 
1.0E-01 & 1.73478908E+01 & 1.23781444E-01 & 1.68873801E+01 & 1.20495587E-01\tabularnewline
\hline 
5.0E-02 & 9.04295359E+00 & 6.45236855E-02 & 1.05889724E+01 & 7.55549079E-02\tabularnewline
\hline 
1.0E-02 & 2.51478744E+00 & 1.79436234E-02 & 1.77559899E+00 & 1.26693331E-02\tabularnewline
\hline 
5.0E-03 & 1.33146975E+00 & 9.50036234E-03 & 7.94414569E-01 & 5.66834224E-03\tabularnewline
\hline 
1.0E-03 & 4.04066751E-01 & 2.88311509E-03 & 1.79227628E-01 & 1.27883296E-03\tabularnewline
\hline 
5.0E-04 & 2.24155836E-01 & 1.59940671E-03 & 8.74185405E-02 & 6.23752667E-04\tabularnewline
\hline 
1.0E-04 & 7.42023262E-02 & 5.29451745E-04 & 1.85142488E-02 & 1.32103693E-04\tabularnewline
\hline 
5.0E-05 & 4.58148569E-02 & 3.26900209E-04 & 9.69012336E-03 & 6.66814607E-05\tabularnewline
\hline 
1.0E-05 & 1.56186641E-02 & 1.11442988E-04 & 1.86120018E-03 & 1.32801185E-05\tabularnewline
\hline 
\end{tabular}
\par\end{centering}

\caption{Error estimation between the exact solution and its regularized solution for both the a priori parameter choice rule and the a posteriori parameter choice
rule in Example 2.\label{tab:1}}
\end{table}
% Figure 01
\begin{figure}[h!]
\begin{center}
	\begin{center}
		\subfigure[$\epsilon=10^{-1}$]{	\includegraphics[scale=0.50]{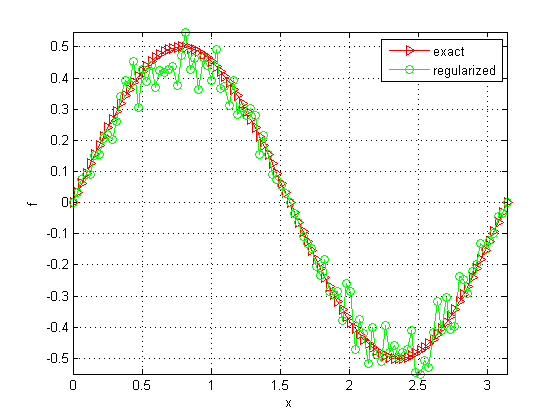} }
		\subfigure[$\epsilon=10^{-2}$]{ \includegraphics[scale=0.50]{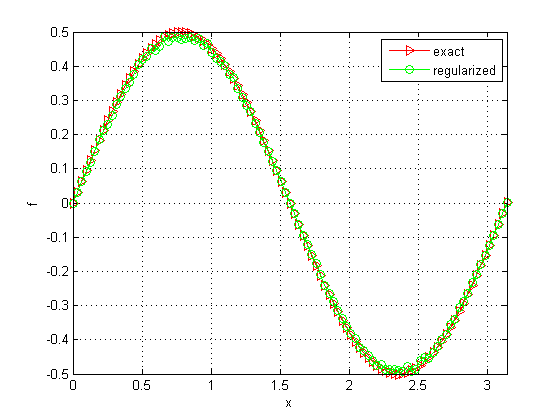} }
	\end{center}
	\begin{center}
		\subfigure[$\epsilon=10^{-3}$]{ \includegraphics[scale=0.50]{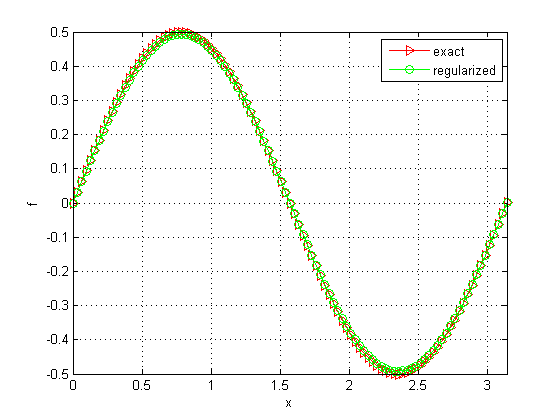} }
		\subfigure[$\epsilon=10^{-4}$]{ \includegraphics[scale=0.50]{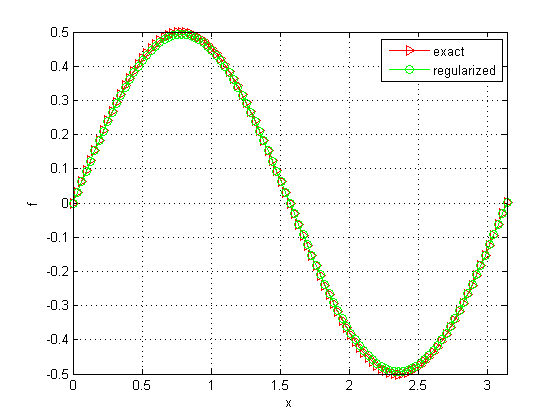} }
	\end{center}
\end{center}
\caption{A comparison between the exact solution and its regularized solution for the a priori parameter choice rule in Example 1}

\label{exp1-fig}
\end{figure}

% Figure 02
\begin{figure}[h!]
\begin{center}
	\begin{center}
		\subfigure[$\epsilon=10^{-1}$]{	\includegraphics[scale=0.50]{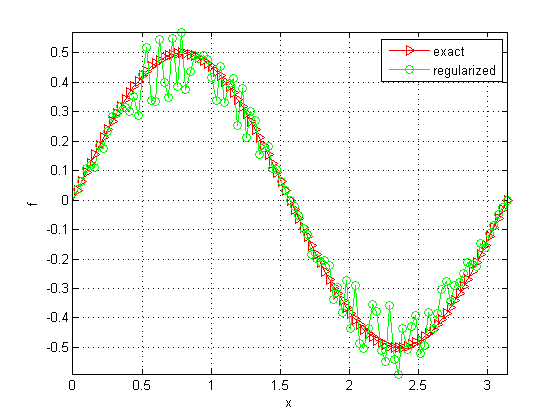} }
		\subfigure[$\epsilon=10^{-2}$]{ \includegraphics[scale=0.50]{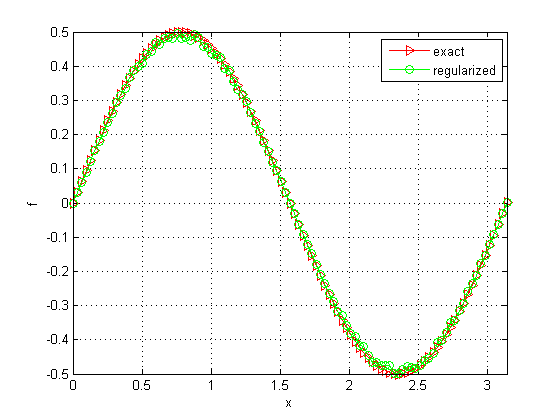} }
	\end{center}
	\begin{center}
		\subfigure[$\epsilon=10^{-3}$]{ \includegraphics[scale=0.50]{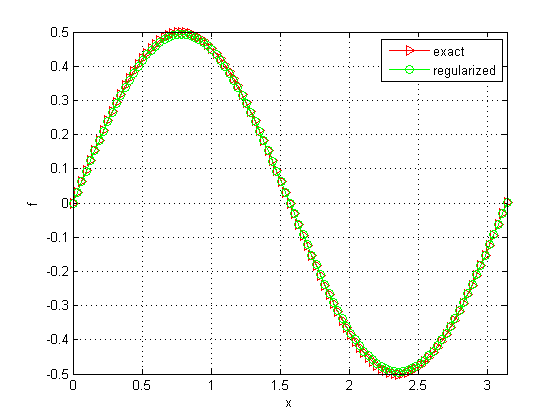} }
		\subfigure[$\epsilon=10^{-4}$]{ \includegraphics[scale=0.50]{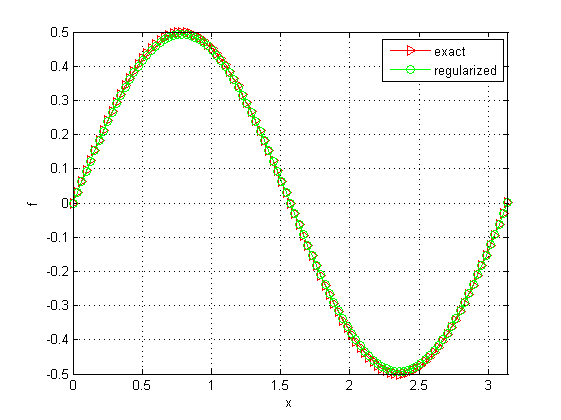} }
	\end{center}
\end{center}
\caption{A comparison between the exact solution and its regularized solution for the a posteriori parameter choice rule in Example 1}
\label{exp1-fig}
\end{figure}

% Figure 03
\begin{figure}[h!]
\begin{center}
	\begin{center}
		\subfigure[$\epsilon=5\times 10^{-1}$]{	\includegraphics[scale=0.50]{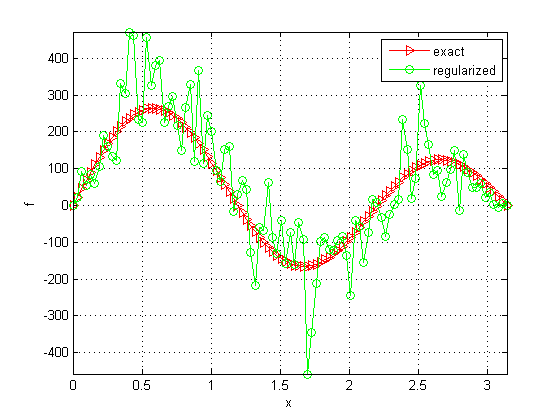} }
		\subfigure[$\epsilon=5\times 10^{-2}$]{ \includegraphics[scale=0.50]{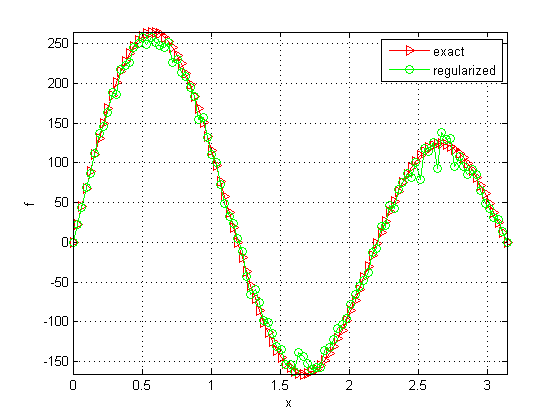} }
	\end{center}
	\begin{center}
		\subfigure[$\epsilon=5\times 10^{-3}$]{ \includegraphics[scale=0.50]{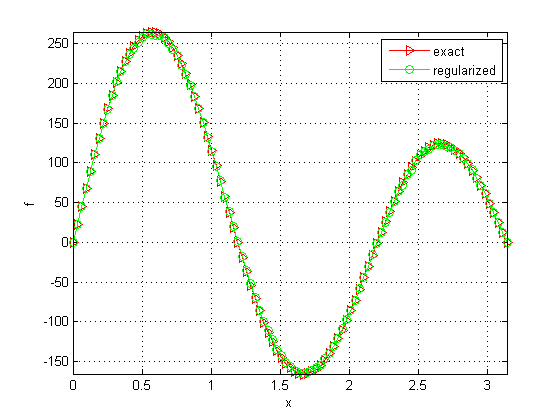} }
		\subfigure[$\epsilon=5\times 10^{-4}$]{ \includegraphics[scale=0.50]{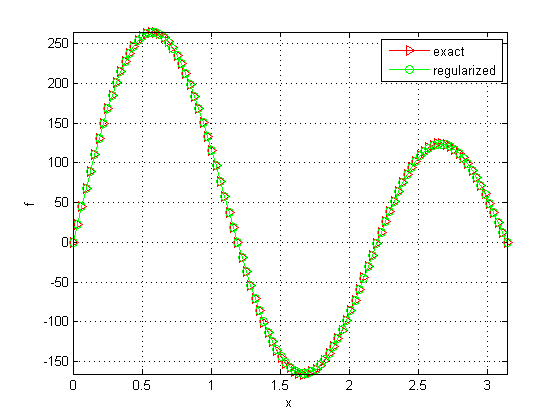} }
	\end{center}
\end{center}
\caption{A comparison between the exact solution and its regularized solution for the a priori parameter choice rule in Example 2}
\label{exp1-fig}
\end{figure}

% Figure 04
\begin{figure}[h!]
\begin{center}
	\begin{center}
		\subfigure[$\epsilon=5\times 10^{-1}$]{	\includegraphics[scale=0.50]{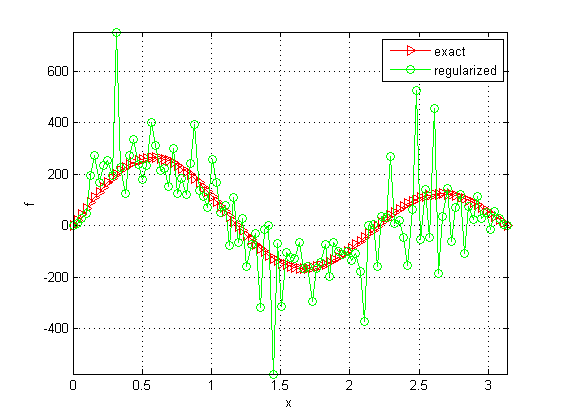} }
		\subfigure[$\epsilon=5\times 10^{-2}$]{ \includegraphics[scale=0.50]{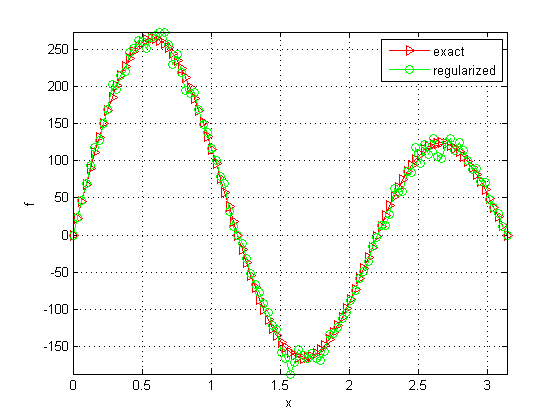} }
	\end{center}
	\begin{center}
		\subfigure[$\epsilon=5\times 10^{-3}$]{ \includegraphics[scale=0.50]{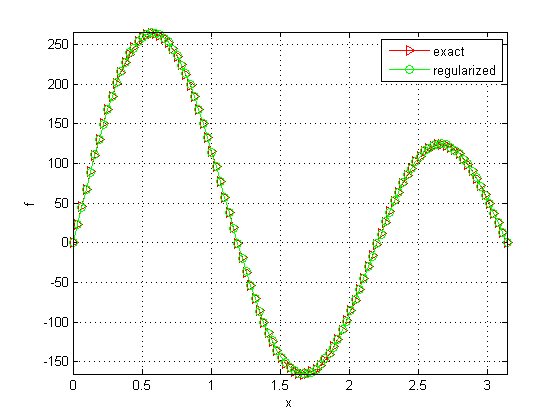} }
		\subfigure[$\epsilon=5\times 10^{-4}$]{ \includegraphics[scale=0.50]{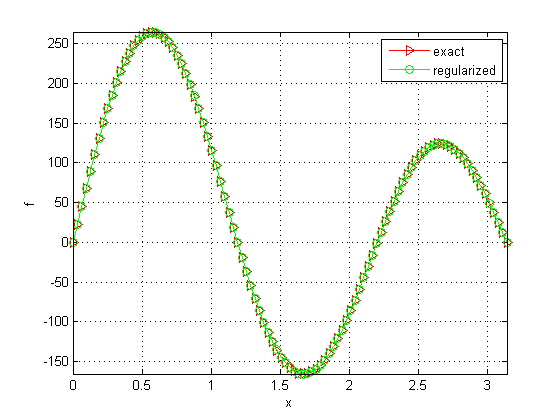} }
	\end{center}
\end{center}
\caption{A comparison between the exact solution and its regularized solution for the a posteriori parameter choice rule in Example 2}
\label{exp1-fig}
\end{figure}


\begin{thebibliography}{10}
\bibitem{AB} J. Atmadja, A.C. Bagtzoglou, \emph{Marching-jury
backward beam equation and quasi-reversibility methods for hydrologic
inversion: Application to contaminant plume spatial distribution recovery}.
WRR 39, 1038C1047 (2003).\label{re1}

\bibitem{key-1} J.R. Cannon, P. Duchateau, \emph{Structural identification
of an unknown source term in a heat equation}, Inverse Problems 14
(1998) 535-551.

\bibitem{CF} Wei Cheng, C.L. Fu, \emph{Identifying an unknown
source term in a spherically symmetric parabolic equation}, Applied
Mathematics Letters, Vol 26, (2013) 387-391.

\bibitem{key-2} M. Denche, K. Bessila, \emph{A modified quasi-boundary
value method for ill-posed problems},\emph{ }J.Math.Anal.Appl, Vol.
301, 2005, pp. 419-426.

\bibitem{Sa} E. G. Savateev, \emph{On problems of determining
the source function in a parabolic equation}, J. Inverse Ill-Posed
Probl. 3 (1995) 83-102.

\bibitem{FL} A. Farcas, D. Lesnic, \emph{The boundary-element
method for the determination of a heat source dependent on one variable},
J. Eng. Math. 54 (2006) 375-388.

\bibitem{Ha} A. Hasanov, \emph{Identification of spacewise and
time dependent source terms in 1D heat conduction equation from temperature
measurement at a final time}. Int. J. Heat Mass Transf. 55, 2069-2080
(2012).

\bibitem{JL} T. Johansson, D. Lesnic, \emph{Determination of a
spacewise dependent heat source}, J. Comput. Appl. Math. 209 (2007)
66-80.

\bibitem{key-4} A. Qian, Y. Li, \emph{Optimal error bound and generalized
Tikhonov regularization for identifying an unknown source in the heat
equation}, J. Math. Chem. 49 (2011), No. 3, 765-775.

\bibitem{key-5} D.D. Trong, N.T. Long, P.N. Dinh Alain, \emph{Nonhomogeneous
heat equation: Identification and regularization for the inhomogeneous
term}, J. Math. Anal. Appl. 312 (2005) 93-104.

\bibitem{key-6} D.D. Trong, P.H. Quan, P.N.D. Alain, \emph{Determination
of a two dimensional heat source: uniqueness, regularization and error
estimate}, J. Comput. Appl. Math. 191 (2006) 50-67.

\bibitem{key-7} L. Yang, C.L. Fu, F.L. Yang, \emph{The method of
fundamental solutions for the inverse heat source problem}, Eng. Anal.
Bound. Elem. 32 (2008) 216-222.

\bibitem{YF} F. Yang, C.L-Fu, \emph{Two regularization methods
for identification of the heat source depending only on spatial variable
for the heat equation}. J. Inverse Ill-Posed Probl. 17 (2009), No.
8, 815-830.

\bibitem{YF1} F. Yang, C.L-Fu, \emph{A simplified Tikhonov regularization
method for determining the heat source}, Applied Mathematical Modelling,
Vol 34, (2010), 3286-3299.

\bibitem{YF2} F. Yang, C.L-Fu, \emph{A mollification regularization
method for the inverse spatial-dependent heat source problem}, J.
Comput. Appl. Math, Vol 255, (2014) 555-567.

\bibitem{key-5} D.D. Trong, N. H. Tuan, \emph{A nonhomogeneous backward
heat problem: Regularization and error estimates}, Electron. J. Diff.
Eqns., Vol. 2008 , No. 33, pp. 1-14.

\bibitem{key-7} Z. Wanga and J. Liu, \emph{Identification of the
pollution source from one-dimensional parabolic equation models},
Appl. Math. Comput. 219 (2012), No. 8, 3403-3413.

\bibitem{Sc} O. Scherzer, \emph{The use of Morozov's discrepancy
principle for Tikhonov regularization for solving nonlinear ill-posed
problems}, Computing, Vol. 51, (1993) pp 45-60. 

\bibitem{CP} D. Coltony, M. Pianayand, R. Potthast, \emph{A simple
method using Morozov’s discrepancy principle for solving inverse scattering
problems}, Inverse Problems \textbf{13} (1997) 1477–1493.

\bibitem{Ki} A. Kirsch, \emph{An Introduction to the Mathematical
Theory of Inverse Problems (Second Edition)}, Applied Mathematical
Sciences 120, Springer, (2011)

\bibitem{Mc} D. McLaughlin, \emph{Investigation of alternative procedures for estimating groundwater basin parameters}, Water Res. Eng., Walnut Creek, Calif., (1975)

\bibitem{Ye} W. W-G. Yeh, \emph{Review of Parameter Identification Procedures in Groundwater Hydrology: The Inverse Problem}, Water Res. Research, Vol. 2(2), (1986)

\bibitem{Ca} J. Carrera, \emph{State of the art of the inverse problem applied to the flow and solute transport problem, in Groundwater Flow and Quality Modeling}, NATO ASI Ser., (1987)

\bibitem{GC} T. R. Ginn and J. H. Cushman \emph{Inverse methods for subsurface flow: A critical review of stochastic techniques}, Stochastic Hydrol. Hydraul., (1990)

\bibitem{Ku} L. Kuiper, \emph{A comparison of several methods for solution of the inverse problem in two-dimensional steady state groundwater flow modeling}, Water Res. Research, Vol. 22(5), (1986)

\bibitem{Su} N.Z. Sun, \emph{Inverse Problems in Groundwater Modeling}, Kluwer Acad., Norwell, Mass. (1994)

\bibitem{MT} D. McLaughlin and L. R. Townley, \emph{A reassessment of the groundwater inverse problem}, Water Res. Research, Vol. 32(5), (1996)

\bibitem{PH} E. P. Poeter and M. C. Hill, \emph{Inverse Models: A necessary next step in groundwater modeling}, Groundwater, Vol. 35(2), (1997)

\end{thebibliography}
\end{document}